\documentclass[11pt]{article}
\usepackage[dvips]{graphicx}
\usepackage{amsmath}
\usepackage{amsthm}
\usepackage{amsfonts,latexsym,amssymb,amsmath, amstext, amsfonts}
\usepackage{verbatim}
\usepackage{hyperref}

\newtheorem{thm}{Theorem}[section]
\newtheorem{lem}[thm]{Lemma}

\newtheorem{dfn}[thm]{Definition}

\newtheorem{rk}[thm]{Remark}

\newtheorem*{con}{Conjecture}
\newtheorem*{thma}{Theorem A}
\newtheorem*{claim}{Claim}

\title{The local $C^1$-density of stable ergodicity}
\author{Yunhua Zhou }

\begin{document}
%\title[The local $C^1$-density of stable ergodicity]{The local $C^1$-density of stable ergodicity}

\date{}

\maketitle

\footnote{Date: \today.}
\footnote{The  author was supported   by NSFC  (11001284) and Natural Science Foundation Project of CQCSTC.}
\footnote{Address:
Department of Mathematics, Chongqing University, Chongqing, 400030,  P. R. China.}
\footnote{{\em E-mail:} zhouyh@cqu.edu.cn}
\vspace{-1.3cm}

\begin{abstract}
The center bundle of a conservative partially hyperbolic diffeomorphism $f$ is called robustly non-hyperbolic if any conservative diffeomorphism which is $C^1$-close to $f$ has non-hyperbolic center bundle.  In this paper, we prove that stable ergodicity is $C^1$-dense among conservative partially hyperbolic systems with robust non-hyperbolic center.


\end{abstract}

\section {Introduction}

Let $M$  be a smooth compact, connected and boundless Riemannian
manifold with dimension $d\geq 3$, and $\mu$ be a smooth volume
measure on $M$ with $\mu(M)=1$. Denote by $\text{Diff}^r_\mu(M)$ the set of
$C^r$   $\mu$-preserving
diffeomorphisms of $M$ endowed  with
$C^r$ topology for $r\geq 1$. If $f\in \text{Diff}^r_\mu(M)$, we also call $f$ is a conservative system.

A diffeomorphism $f:M\to M$ is said to be {\em partially hyperbolic}, if $f$
admits a non-trivial $Df$-invariant splitting of the tangent bundle
$TM=E^s\oplus E^c\oplus E^u$ and numbers $0<\alpha_s<\alpha'_c\leq \alpha''_c<\alpha_u$
 such that $\alpha_s<1<\alpha_u$ and for any $x\in M$, we have
$$\|Df|_{E^s(x)}\|<\alpha_s,\ \alpha'_c\leq  m(Df|_{E^c(x)}),\ \|Df|_{E^c(x)}\|\leq \alpha''_c {\text{ and }} \alpha_u<m(Df|_{E^u(x)}),$$
where $m(Df|_{E})$ is the minimum norm of $Df|_{E}$, i.e.,
 $$m(Df|_{E})= \inf\{\|Dfv\|: v \in E, \|v\|=1\}.$$

The subbundles $E^u, E^c$ and $ E^s$ are called {\em unstable,
center} and {\em stable bundle}. Set $\beta=\dim (E^\beta)$ for $\beta=s,c,u$. Partial hyperbolicity is a robust property. That is to say, for any given partially hyperbolic diffeomorphism  $f$ of $M$, there is  a $C^1$ neighborhood $\mathcal{U}$ of $f$ in $\text{Diff}^1(M)$ such that any $g\in \mathcal{U}$ is partially hyperbolic.
 We denote by $\text{PH}^r_\mu(M)$ the family of $C^r$
conservative partially hyperbolic diffeomorphisms of $M$ endowed with
$C^r$ topology for $r\geq 1$.
 Given $f\in \text{PH}^1_\mu(M)$, the center bundle $E^c_f$ of $f$ is called {\em robustly non-hyperbolic} if there is a $C^1$ neighborhood $\mathcal{U}$ of $f$ in $\text{PH}^1_\mu(M)$ such that each $g\in \mathcal{U}$ has two ergodic measures $\mu_1$ and $\mu_2$ satisfy $\lambda^+_{\mu_1}\leq 0$ and  $\lambda^-_{\mu_2}\geq 0$, where $\lambda^+_{\mu_1}$ and  $\lambda^-_{\mu_2}$ are the largest and smallest Lyapunov exponents of $\mu_1$ and $\mu_2$ in $E^c_g$.

We set $$\mathcal{P}=\{f\in \text{PH}^1_\mu(M):\ E^c_f \text{ is robustly non-hyperbolic}\}.$$
Then $\mathcal P$ is a non-empty open subset of $\text{PH}^1_\mu(M)$. The openness is obvious by the definition. On the other hand, if
a conservative partially hyperbolic system $f$ have two hyperbolic periodic points with indices $s$ and $s+c$ respectively, then
$f\in \mathcal P$. This implies that $\mathcal P$ is non-empty.

A diffeomorphism $f\in \text{Diff}^1_\mu(M)$ is {\em ergodic} (with respect to $\mu$)  if only full or null $\mu$-volume sets are invariant
under it. $f$
 is  {\it stably ergodic} if  there exists a $C^1$ open neighborhood $\mathcal U$ of $f$ in
$\text{Diff}^1_\mu(M)$ such that any diffeomorphism $g\in \mathcal U\cap \text{Diff}^2_\mu(M)$ is
ergodic with respect to $\mu$.

The main result of this paper is
\begin{thma}\label{thm}
There is a subset $\mathcal{D}$ of $\mathcal{P}$ such that $\mathcal{D}$ is $C^1$-dense in $\mathcal {P}$ and each $f\in \mathcal{D}$ is stable ergodic.
\end{thma}

The study of stable ergodicity has a long-time history. In \cite{A,AS},  by using
 Hopf Argument (\cite{HO}),
D. Anosov and J. Sinai  established ergodicity of all
$C^2$ volume-preserving uniformly hyperbolic systems (Anosov
systems), including geodesic flows for compact manifolds of negative
sectional curvature.
 In 1994, M. Grayson, C. Pugh and M. Shub
(\cite{GPS}) gave the first nonuniformly hyperbolic example of a
stably ergodic system. These systems are partially hyperbolic.
Following this direction, Pugh and Shub believe  that a little
hyperbolicity goes a long way in guaranteeing ergodicity and,in
\cite{PS,PS2}, they posed the following Stable Ergodicity
Conjecture:

\begin{con}
 Stable ergodicity is $C^r$-dense among conservative partially hyperbolic diffeomorphisms.
 \end{con}

%We mention that there are  examples of conservative stably ergodic
%systems which are not partially hyperbolic (\cite{T}).

At the same time, Pugh and Shub gave a program to deal with this
conjecture: they conjectured that stable accessibility is dense and
essential accessibility implies ergodicity among volume-preserving,
partially hyperbolic diffeomorphisms. In recent years, many advances
have been made for this conjecture (e.g. see the survey
\cite{BPSW,HHU}). For example, F. Rodriguez Hertz, M. Rodriguez
Hertz and R. Ures (\cite{HHU2}) proved that stable ergodicity is
$C^\infty$-dense among partially hyperbolic diffeomorphisms with
one-dimensional center bundle; K. Burns and A. Wilkinson (\cite{BW})
proved that essential accessibility implies ergodicity if the system
is center bunched, and C. Bonatti, C. Matheus, M. Viana, and
A.Wilkinson (\cite{BMVW}) proved the conjecture in the $C^1$
topology for one-dimensional center bundle.

As pointed in \cite{HHTU, W2}, many arguments of previous works
(such as \cite{BW} and \cite{HHU2}) seem to be hard to generalize
and have reached their limits in these directions. Recently, a new
alternate criterion to establish ergodicity be obtained by F.
Rodriguez Hertz, M. Rodriguez Hertz, A. Tahzibi and R. Ures in
\cite{HHTU3,HHTU}. Using this argument, the authors proved the Pugh
and Shub's Conjecture with two-dimensional center bundle in
$C^1$-topology.

 Highly motivated by the Stably Ergodic Conjecture, our main result (Theorem A) of this paper provides a large class of conservative
partially hyperbolic diffeomorphisms  which can be $C^1$ approximated
by stably ergodic systems. Unlike \cite{BW} or \cite{HHTU}, these
systems considered here are more general and the center dimension
is not necessarily two.

\section {Preliminaries}\label{sec:pre}

Given $f\in \text{Diff}^1_\mu(M)$, by Oseledec Theory (\cite{O}), there is a
$\mu$-full invariant set $\mathcal{O}\subset M$ such that for every $x\in
\mathcal {O}$ there exist a splitting (which is called {\em Osledec
splitting})
$$T_xM= E_1(x)\oplus \cdots\oplus  E_{ k(x)}(x)$$
and real numbers (the Lypunov exponents of $\mu$) $\chi_1(f,x)<\chi_2(f,x)<\cdots<
\chi_{k(x)}(f, x)$ satisfying $Df( E_j(x))= E_j(fx)$ and
$$\lim\limits_{n\to \pm \infty}\frac{1}{n}\ln \|Df^nv\|=\chi_j(f,x)$$
for every $v\in  E_j(x)\setminus \{0\}$ and $j=1, 2, \cdots,  k(x)$.
 In the following,  by counting multiplicity, we also rewrite the Lyapunov
exponents of $\mu$ as $$\lambda_1(f,x)\leq \lambda_2(f,x)\leq \cdots \leq \lambda_d(f,x).$$
For $i=1,2,\cdots,d$, define
$$LE_i(f)=\int_M(\lambda_1(f,x)+\cdots+\lambda_i(f,x))d\mu.$$
It is obvious that the continuous points of the {\em Lyapunov map}
$$f\in \text{Diff}^1_\mu(M)\mapsto (LE_1(f),\cdots,LE_{d-1}(f))\in \mathbb{R}^{d-1}$$
is a residual set $\mathcal{R}_0$ of $f\in \text{Diff}^1_\mu(M)$.

For $f\in \text{PH}^1_\mu(M)$, the distributions $E^u$ and $E^s$ are
integrable and their integrable manifolds form two transversal
 foliations of $M$, the {\em strongly stable} and {\em strongly unstable foliations} of $M$, which we denote by $\mathcal{W}^u$ and $\mathcal{W}^s$ respectively.
 For every $x\in M$ the leaves $\mathcal{W}^u(x)$ and $\mathcal{W}^s(x)$ of the foliations containing $x$ are smooth immersed submanifolds in $M$
 called the {\em strong unstable} and {\em strong stable global manifolds} at $x$ (see e.g. \cite{BP,HPS}).

Two points $x, y\in M$ are called {\em accessible} if there are
points $x=z_0, z_1, \cdots, z_{l-1}, z_l=y, z_i\in M$ such that
$z_i\in \mathcal{W}^u(z_{i-1})$ or $z_i\in \mathcal{W}^s(z_{i-1})$ for $i=1, \cdots, l.$
A diffeomorphism $f$ is called an {\em accessible diffeomorphism} if
it has the accessibility property, i.e.,  any pare points $x,y\in M$
are accessible. $f$ is {\em essentially accessible} if there are
$\mu$-full measure subset $M'\subset M$ such that  any pare points
$x, y\in M'$ are accessible.
 $f$ is {\em stably accessible} if there is a $C^1$
neighborhood of $f$ composed by accessible diffeomorphisms.

Accessibility is important to show the ergodicity of partially hyperbolic diffeomorphisms. In \cite{DW}, D. Dolgopyat and A. Wilkinson proved that stable accessibility is $C^1$ dense. That is

\begin{lem}\label{thm:DW}
{\em (\cite{DW})} There is a $C^1$ open and dense set $\mathcal{R}_1$ in $\text{\em PH}^r_\mu(M)$ ($r\geq 1$) such that each  $f\in \mathcal{R}^1$ is accessible.
\end{lem}

The following lemma  can be find in \cite{BDP}.

\begin{lem}\label{lem:3}
Let $f\in \text{{\em PH}}^2_\mu(M)$ and $f$ is accessible. Then almost every orbit is dense in $M$.
\end{lem}

An ergodic measure $\nu$ of $f\in \text{Diff}^r(M)$ is called {\em hyperbolic} if all the Lyapunov exponents of $\nu$ is not zero. If $r>1$, for every point $x\in \mathcal{O}$, there are {\em Pesin's stable} and {\em unstable} manifolds which we denote by $W^s(x)$ and $W^u(x)$. If $f$ is also partially hyperbolic, we have $\mathcal{W}^s(x)\subset W^s(x)$ and  $\mathcal{W}^u(x)\subset W^u(x)$.

Given a diffeomorphism  $f$ and a $f$-invariant set $K\subset M$, a
$Df$-invariant splitting
 $T_xM=E(x)\oplus F(x)$
($x\in K$) of the tangent bundle over $K$ is $l$-{\em dominated} if
for any  $x\in K$ one has
$$\frac{\|D_xf^l|_{E(x)}\|}{m(D_xf^l|_{F(x)})}<\frac{1}{2}$$
and the dimension of $E(x)$ is independent of $x\in K$. We denote the domination by $E\prec_l F$ and call $\dim (E)$ the {\em index} of the domination.
A
$Df$-invariant splitting
 $T_xM=E_1(x)\oplus \cdots\oplus E_k(x)$
($x\in K$) of the tangent bundle over $K$ is $l$-{\em dominated} if
for any $i<j$, one has  $E_i\prec_l E_j$ for some $l$.

Dominated splitting is
unique, transverse  and continuous. Moreover, the dominated
splitting has some robust properties (see e.g. \cite{BDV2}).  A dominated splitting $E_1\oplus\cdots\oplus E_k$ is called the {\em finest dominated splitting} if there is no dominated splitting in each invariant bundle $E_i$ for all $i=1, \cdots, k.$ Moreover, a splitting is called the {\em robust finest dominated splitting} if the continuation of the   splitting is the finest dominated splitting of the $C^1$ perturbation diffeomorphism.

For a hyperbolic periodic point $P$ of $f$, we denote
by  $ind(P)$ the {\em index} of $P$, where the index of $P$ refers to the dimension of the stable bundle of $P$.
The {\em homoclinic class} of a hyperbolic saddle $P$
of a diffeomorphism $f$ , denoted by $H(P,f )$, is the closure of the transverse intersections
of the invariant manifolds (stable and unstable ones) of the orbit of $P$.

\begin{lem}\label{thm:connectin}
{\em (\cite{BC})} There exists a residual set $\mathcal{R}_2$ of $\text{\em Diff}^1_\mu(M)$   such that all diffeomorphisms in $\mathcal{R}_2$ is transitive.
Moreover, $M$ is the unique homoclinic class.
\end{lem}

Connecting Lemma was firstly proved by S. Hayashi (\cite{Ha}) and was extended by L. Wen and Z. Xia in \cite{WX} (see also \cite{Ar}). The following Connecting Lemma which established by C. Bonatti and S. Crovisier from the proof of Hayashi's will permit us to create intersections between stable and unstable manifolds.

\begin{lem}\label{thm:BC}
{\em (Connecting Lemma \cite{BC})} Let $Q,P$ be hyperbolic periodic points of a $C^r$ $(r\geq 1)$ transitive diffeomorphism preserving a smooth measure $\mu$. Then, there exists a $C^1$-perturbation $g\in C^r$ preserving $\mu$ such that $W^s(P)\cap W^u(Q)\neq \emptyset.$
\end{lem}

Blender has been introduced firstly in \cite{BD}, and it is a very
useful tool to understand the dynamical properties such as the
transitivity and ergodicity (e.g. see \cite{BD,BD2,HHTU2,HHTU}).  There are several definitions of blender(\cite{BD,BDV2,LSY,HHTU2}). The following definition comes from \cite{BD2} and \cite{HHTU2}.

\begin{dfn}\label{dfn:blender}
Let $P, Q$ be hyperbolic periodic points of a diffeomorphism $f$ with index $i$ and $i+1$ respectively. We say that $f$ has a $cu$-blender of index $i$ associated to $(P,Q)$ if

(1) $P$ is a partially hyperbolic periodic point of $f$ such that $Df$ is expending on $E^c$ and $\dim (E^c)=1$;

(2) there is a small open set $\text{\em Bl}^u(P)$ such that every $(d-i)$-strip well placed in $\text{\em Bl}^u(P)$ transversely intersects $W^s(P)$;

(3) $W^u(Q)\cap \text{\em Bl}^u(P)$ contains a vertical disk $D$ through $\text{\em Bl}^u(P)$, i.e., $D$ is a $(d-i-1)$-disk which is centered at a point in $\text{\em Bl}^u(P)$, the radius of $D$ is much bigger than the radius of $\text{\em Bl}^u(P)$ and $D$ is almost tangent to $E^u$;

(4) this property is $C^1$-robust.

Define a $cs$-blender in an analogous way by concerning $f^{-1}$.
\end{dfn}

We also will use the following two lemmas.

\begin{lem}\label{thm:AB} {\em (Theorem C of \cite{AB})}
There is a residual set $\mathcal{R}_3\subset \text{\em Diff}^1_\mu(M)$ such that for every $f\in \mathcal{R}_3$, every $f$-invariant Borel set $\Lambda\subset M$, and every $\eta>0$, if $g\in \text{\em Diff}^1_\mu(M)$ is sufficiently close to $f$ then there exists a $g$-invariant Borel set $\widetilde{\Lambda}$ such that $\widetilde{\Lambda}\subset B_\eta(\Lambda)$ and $\mu(\widetilde{\Lambda}\triangle \Lambda)<\eta$, where $ B_\eta(\Lambda)=\{x\in M: \rho(x,y)<\eta \text{ for some }y\in \Lambda\}$ and $\rho$ is the distance on $M$ induced by the Riemannian metric.
\end{lem}

\begin{lem}\label{thm:AV}{\em (\cite{Av})}
$C^\infty$ diffeomorphisms are dense in the set of $C^1$ diffeomorphisms preserving $\mu$.
\end{lem}

\section[Proof of Theorem 1.1]{Proof of Theorem \ref{thm}}
We first give a lemma which is important to the proof of Theorem A.

\begin{lem}\label{lem:1}
There is a residual subset $\mathcal{R}_4$ of $\text{{\em Diff}}^1_\mu(M)$ such that for every $f\in \mathcal{R}_4$, $M$ is the unique homoclinic class. Moreover, if there are two hyperbolic saddles of indices $\alpha$ and $\beta$, then $M$ contains a dense subset of saddles of index $\tau$ for all
$\tau\in [\alpha,\beta]\cap \mathbb{N}$.
\end{lem}

{\bf Proof.} This is the direct result of Lemma \ref{thm:connectin} and the conservative version of Theorem 1.1 of \cite{ABCDW}(or see the Lemma 3.8 of \cite{LSY}).
\hfill$\Box$

Now, we  recall the criteria of ergodicity of \cite{HHTU}.
Given a  diffeomorphism $f$ and a hyperbolic periodic point $P$, we define two invariant sets:
$$\Lambda^u(P)=\{x\in \mathcal{O}: W^s(P)\pitchfork W^u(x)\neq \emptyset\},\ \ \Lambda^s(P)=\{x\in \mathcal{O}: W^u(P)\pitchfork W^s(x)\neq \emptyset\},$$
where $\mathcal{O}$ is the set of Oseledec regular points and $W^s$ ($W^u$)
is the Pesin global stable (unstable) manifold.

\begin{lem}\label{thm:hhtu}
{\em (Theorem A of \cite{HHTU})} Let $f\in \text{\em Diff}^r_\mu(M)$ for $r>1$. If $\mu
(\Lambda^s)>0$ and  $\mu (\Lambda^u)>0$ for some hyperbolic periodic
point $P$, then
$$\Lambda(P):=\Lambda^s(P)\cap \Lambda^u(P)= \Lambda^s(P)= \Lambda^u(P)\ (\text{\em mod}\ 0)$$
and $f$ is ergodic on $\Lambda(P)$. Moreover, $f$ is non-uniformly
hyperbolic on $\Lambda(P)$.

\end{lem}

\begin{rk}\label{rk:hhtu}
It is obvious  that if $\Lambda^s(P)\cup \Lambda^u(P)= M (\text{\em mod}\ 0)$, then $f$
is ergodic respect to $\mu$.
\end{rk}

In the following Lemma, we give a dense subset $\mathcal{D}\subset \text{PH}^1_\mu(M)$. In fact, to prove Theorem A, we only need to prove that the stable ergodicity can $C^1$ approximate to each system of a dense subset of $\mathcal{P}$.

\begin{lem}\label{pro2}
There is a $C^1$ residual subset $\mathcal{D}$ in $\text{\em{PH}}^1_\mu(M)$ such that for any
$f\in \mathcal{D}$, we have

(1) $f$ is stably accessible  and

(2) there exists a robust finest dominated splitting of $Df$,
\begin{equation}\label{eq:1}
 TM=E_{1}\oplus
E_{2}\oplus \cdots\oplus E_{k},
\end{equation} such that  the Lyapunov exponents at $x$ in $E_i$ are equal for $\mu$-a.e. $x\in M$ and all $i=1,2,\cdots,k$.

\end{lem}

\begin{proof}
By Lemma 2.2 of \cite{LSY}, there is an open and dense subset  $\mathcal{D}_1\subset \text{PH}^1_\mu(M)$ such that, for each $f\in \mathcal{D}_1$, $Df$ has a robust finest dominated splitting $ TM=E_{1}\oplus
E_{2}\oplus \cdots\oplus E_{k}$. By Lemma \ref{thm:DW} and \ref{thm:connectin}, there is a residual set $\mathcal{D}_2\subset\mathcal{D}_1\subset \text{PH}^1_\mu(M)$ such that each $f\in \mathcal{D}_2$ is stably accessible and $M$ is the unique homoclinic class. Set $\mathcal{D}=\mathcal{D}_2\cap \mathcal{R}_3$, where $\mathcal{R}_3$ refers to Lemma \ref{thm:AB}. We shall prove that $\mathcal{D}$ satisfies the lemma.

For any $1\leq i< d$ and $l\in \mathbb{N}$, denote $D_i(f,l)$ by the set of points $x$ such that there is a $l$-dominated splitting of index $i$ along the orbit of $x\in M$. Then $D_i(f,l)$ is a compact invariant set. Set $$\Gamma_i(f,l)=M\setminus D_i(f,l) \text{ and } \Gamma_i(f, \infty)=\bigcap\limits_{i=1}^\infty \Gamma_i(f,l), \text{ for } i=1,2, \cdots, d-1.$$
We shall show that, up to zero measure, either $\Gamma_i(f, \infty)=M$ or $D_i(f, l)=M$ for some $l$. In fact, if $\mu(D_i(f, l))=0$ for all $l$, then  $\Gamma_i(f, \infty)=M \text{ mod } 0$. If $\mu(D_i(f, l))>0$ for some $l$, by Lemma \ref{thm:AB} and \ref{thm:AV}, for any $\eta>0 $ there is $g\in\text{Diff}^2_\mu(M)$ which is $C^1$ close to $f$ and a $g$-invariant Borel set $\widetilde{\Lambda}$ such that $\widetilde{\Lambda}\subset B_\eta(D_i(f,l))$ and $\mu(\widetilde{\Lambda} \triangle D_i(f,l))<\eta.$ Since $f$ is stably accessible and $g$ is $C^2$, Lemma \ref{lem:3} implies that $\widetilde{\Lambda}$ is dense in $M$. So,  $D_i(f,l)$ is $\eta$-dense in $M$. By the arbitrary of $\eta$, $D_i(f,l)$ is dense in $M$. So we have $D_i(f,l)=M$ since $D_i(f,l)$ is compact.

Noting that (\ref{eq:1}) is finest dominates splitting, $D_i(f,l)=M$ if and only if $i=\dim (E_1)+\cdots+\dim(E_j)$ for some $j=1, \cdots, k$.
By \cite{BV}, for $\mu$-a.e. $x\in M$ and  any $i=1,2,\cdots, k$, the Lyapunov exponents in $E_i$ are equal.
\end{proof}

\begin{proof}[Proof of Theorem A]
For any $f\in \text{PH}^1_\mu$, $r\geq 1$ and $\varepsilon>0$, we set
$$\mathcal{U}^r(f,\varepsilon)=\{g\in \text{PH}^r_\mu: \ g \text{ is } \varepsilon\text{-}C^1\text{-}\text{close to } f\}.$$
To prove Theorem A, we only need to prove that for any $f\in \mathcal{D}\cap\mathcal{P}\cap \mathcal{R}_0$ and any $\varepsilon>0$, there is $g\in \mathcal{U}^2(f,\varepsilon)$ such that $g$ is stably ergodic.

Since $f\in \mathcal{P}$, there is an ergodic measure $\mu_1$ such that $\lambda^+_{\mu_1}\leq 0,$ where $\lambda^+_{\mu_1}$ is the largest and smallest Lyapunov exponent of $\mu_1$ in $E^c_f$.
By Ergodic Closing Lemma (\cite{Ar}), $\mu_1$-a.e. point is well closable. Then, for any $\epsilon_1\leq \frac{\varepsilon}{3},$ there are $f'\in \mathcal{U}^1(f, \frac{\epsilon_1}{2})$ and a periodic point $P'$ of $f'$ such that $\lambda_{s+c}(P')\leq \frac{\epsilon_1}{3}$. If $\lambda_{s+c}(P')<0$, then $P'$ is a hyperbolic periodic point with index $s+c$. Otherwise,   using the conservative version of Frank's Lemma, one can get a new diffeomorphism $f''\in \mathcal{U}^1(f', \frac{\epsilon_1}{2})$  which  has the periodic point $P''$ with index $s+c$. Anyway, for the $\epsilon_1$, there is $f_1\in \mathcal{U}^1(f, \epsilon_1)$ such that $f_1$ has a hyperbolic periodic point $P_1$ with index $s+c$.

Since $f$ has a robustly non-hyperbolic center bundle, if we select $\epsilon_1$ small enough, there is an ergodic measure $\nu$ of $f_1$ such that $\lambda^-_\nu\geq 0$, where $\lambda^-_{\nu}$ is the smallest  Lyapunov exponent of $\nu$ in $E^c_{f_1}$. By the similar discussion as above, for any $\epsilon_2\leq \epsilon_1$, there is $f_2\in \mathcal{U}^1(f_1, \epsilon_2)$ such that $f_2$ has a hyperbolic periodic point $Q$ with index $s$. Since $P_1$ is a hyperbolic periodic point of $f_1$, the continuation $P$ of $P_1$ is a hyperbolic point of $f_2$ and has the same index as the $P_1$'s if $\epsilon_2$ is small enough.

That is to say, for any $\epsilon_1>0$, there  is $f_2\in \mathcal{U}^1(f, 2\epsilon_1)$ such
that $f_2$ has two hyperbolic periodic points $P$ and $Q$ with indices $s+c$ and $s$ respectively.

 If $\epsilon_3\in(0, \epsilon_2)$ is small enough, any $g\in \mathcal{U}^1(f_2, \epsilon_3)$  has two hyperbolic periodic points of indices $s$ and $s+c$ by the hyperbolicity of $P$ and $Q$. Moreover,  by Lemma \ref{lem:1}, any $g\in \mathcal{U}^1(f_2, \epsilon_3)\cap \mathcal{R}_4$ has a dense subset of saddles of index $i$ for all $i\in \{s, s+1, \cdots, s+c\}$. So, there is an open set $\mathcal{V}_0\subset \mathcal{U}^1(f_2, \epsilon_3)$ such that each $g\in \mathcal{V}_0$ has a subset of hyperbolic periodic points of index  $i$ for all $i\in \{s, s+1, \cdots, s+c\}$.

\begin{lem}\label{lem:5}
If there are two partially hyperbolic points $P',Q'$ with indices $i,i+1$ and one-dimension center,
then there exists a $cu$-blender of index $i$ associated to $(P,Q)$ and $P,Q$ are homoclinic related to $P',Q'$
respectively.
\end{lem}

\begin{proof}
This is the conservative version of  subsection 4.1 of \cite{BD2}.  One also can  see the construction in \cite{HHTU2}.
\end{proof}

Now we continue to prove the Theorem. Selecting a diffeomorphism $g_0\in \mathcal{V}_0$, since $g_0$ has two saddles of indices $s+c-1$ and $s+c$, by Lemma \ref{lem:5} and the robust property of blender, there is   an open subset $\mathcal{V}_1\subset  \mathcal{V}_0$, such that any $g\in \mathcal{V}_1$ has a $cu$-blender of index $s+c-1$.  Selecting a diffeomorphism $g_1\in \mathcal{V}_1$, since $g_1$ has two saddles of indices $s+c-2$ and $s+c-1$, by Lemma \ref{lem:5} and the robust property of blender again, there is   an open subset $\mathcal{V}_2\subset  \mathcal{V}_1$, such that any $g\in \mathcal{V}_2$ has a $cu$-blender of index $s+c-2$. Noting that $ \mathcal{V}_2\subset \mathcal{V}_1$, $g$ also has a  $cu$-blender of index $s+c-1$. Inductively, we obtain open sets $\mathcal{V}_c\subset \cdots\subset \mathcal{V}_1$ such that for any $1\leq i\leq c$ and any $g\in \mathcal{V}_i$, $g$ has $i$ $cu$-blenders of indices $s+c-1, s+c-2,\cdots, s+c-i$ respectively. Especially, for any $g\in\mathcal{V}_c$ and any $s\leq i<s+c$, there is $cu$-blender of index $i$ associated to $(P_{i,g}, Q_{i+1,g})$.

To continue the proof, we need the following lemma.

\begin{lem}\label{lem:2}
If $\varepsilon$ is small enough, then there is  $g\in\mathcal{V}_c\cap \text{\em Diff}^2(M)$ such that $g$ is stably  ergodic.
\end{lem}

\begin{proof} Firstly, if $$\int_M\lambda_{s+c}(f,x)d\mu<0\ (\text{or  }\int_M\lambda_{s+1}(f,x)d\mu>0),$$ then for $\varepsilon$ small enough, we have
$$\int_M\lambda_{s+c}(g, x)d\mu<0\ (\text{or  }\int_M\lambda_{s+1}(g, x)d\mu>0), \ \ \forall g\in\mathcal{V}_c$$ since $f$ is the continuous point of Lyapunov map.
 Then, by Theorem 4 (or section 1.8) of \cite{BDP}, each $g\in\mathcal{V}_c\cap \text{Diff}^2(M)$ is stably ergodic.

Otherwise, there are $1\leq t\leq c$ and $1\leq\kappa<t$ such that
$$\int_M\ln|\det (Df|_{E^{c_\kappa}_f})|d\mu<0 \text{ and }\ \int_M\ln|\det (Df|_{E^{c_{\kappa+1}}_f})|d\mu>0,$$
where $E^c_f=E^{c_1}\oplus\cdots \oplus E^{c_t}$ is robustly finest dominated splitting.
Then,
\begin{equation}\label{eq:3}
\int_M\ln|\det (Dg|_{E^{c_\kappa}_g})|d\mu<0 \text{ and }\ \int_M\ln|\det (Dg|_{E^{c_{\kappa+1}}_g})|d\mu>0
\end{equation}
for any $g\in \mathcal{V}_c$ if $\varepsilon$ is small enough.

We Choose a diffeomorphisim $h\in \mathcal{V}_c\cap \mathcal{D}\cap \mathcal{R}_0$. By Lemma \ref{pro2},
$$\int_M\ln|\det (Dh|_{E^{c_\tau}_h})|d\mu=\dim(E^{c_\tau}_g)\cdot \int_M\lambda_{s+\cdots+c_\tau}(h, x)d\mu,\ \text{for } \tau=\kappa, \kappa+1. $$
Then
 (\ref{eq:3}) implies that
$$\int_M\lambda_{s+\cdots+c_\kappa}(h, x)d\mu<0\ \text{ and }\ \int_M\lambda_{s+\cdots+c_\kappa+1}(h,x)d\mu>0.
$$
Noting that $h$ is the continuous point of Lyapunov map, we have
\begin{equation}\label{eq:2}
\int_M\lambda_{s+\cdots+c_\kappa}(g, x)d\mu<0\ \text{ and }\ \int_M\lambda_{s+\cdots+c_\kappa+1}(g,x)d\mu>0.
\end{equation}
for some small neighborhood $\mathcal{V}'_c$ of $h$ and any $g\in\mathcal{V}'_c$.
On the other hand, by Lemma \ref{thm:BC}, there is an open set $\mathcal{V}\subset \mathcal{V}'_c$ such that for any $g\in \mathcal{V}$ and any $s< i\leq s+c$, $P_{i,g}$ is homoclinic related to $Q_{i,g}$.  Taking $g\in \mathcal{V}\cap \text{Diff}^2(M)$ , we shall prove $g$ is ergodic (and hence is stably ergodic) in the following.

Set
$$\Lambda^+=\{x\in \mathcal{O}: \lambda_{s+\cdots+c_\kappa+1}(x)>0\},\ \ \Lambda^-=\{x\in \mathcal{O}: \lambda_{s+\cdots+{c_{\kappa}}}(x)<0\}.$$
By (\ref{eq:2}), $\mu(\Lambda^+)\cdot \mu(\Lambda^-)>0$. Moreover, we have
$\mu(\Lambda^+\cup\Lambda^-)=1$ because of the domination
$E^{c_{\kappa}} \prec E^{c_{\kappa+1}}$. If we have showed that $\Lambda^+\subset \Lambda^u$ (mod $0$) and  $\Lambda^-\subset \Lambda^s$ (mod $0$), then $g$ is ergodic by Lemma \ref{thm:hhtu} and Remark \ref{rk:hhtu}.
We  only prove the first part and the proof of the second part is similar.

In fact, we  shall prove that for almost every point $x\in \Lambda^+$, there is $y\in orb(x)$ such that $y\in \Lambda^u$. Then, by the invariance of $\Lambda^u$, we have $x\in \Lambda^u$.

Since $g$ is accessible, by Lemma \ref{lem:3}, there is a $\mu$-full measure set $\mathcal{O}'$ such that $orb(x)$ is dense in $M$ for every $x\in \mathcal{O}'$. Recall that $g$ has a $cu$-blender of index $s+c-1$ associated to $(P_{s+c-1}, Q_{s+c})$. So, given $x\in \Lambda^+\cap  \mathcal{O}'$, there is $y\in orb(x)$ such that $y\in \text{Bl}^u(P_{s+c-1})$. By the Pesin's Stable Manifold Theorem,
$y$ has unstable manifold $W^u(y)$ of dimension $c_{\kappa+1}+\cdots+c_t+u$ and $W^u(y)$ is tangent to the bundle $E^{\kappa+1}\oplus\cdots\oplus E^u$. Moreover, $W^u(y)$ should intersect transversely  the stable manifold of $P_{s+c-1}$ (which has index $s+c-1$) since  the strong unstable manifold $\mathcal{W}^u(y)$ has uniform size and  $y\in \text{Bl}^u(P_{s+c-1})$. That is to say,  $W^u(y)\pitchfork W^s(P_{s+c-1})\neq \emptyset$.
Since $P_{s+c-1}$ is homoclinic related to $Q_{s+c-1}$, we have $W^u(y)\pitchfork W^s(Q_{s+c-1})\neq \emptyset$ by $\lambda$-Lemma.
\begin{claim}
$W^u(y)\pitchfork W^s(P_{s+c-2})\neq \emptyset$.
\end{claim}
\begin{proof}[Proof of Claim]
By the definition of blender, $W^u(Q_{s+c-1})\cap \text{Bl}^u(P_{s+c-2})$ contains a vertical disk $D$ through $\text{ Bl}^u(P_{s+c-2})$.
On the other hand, since  $W^u(y)\pitchfork W^s(Q_{s+c-1})\neq \emptyset$, by the $\lambda$-Lemma, $g^n(W^u(y))$  contains a $(u+1)$-dimension manifold closing to $W^u(Q_{s+c-1})$ for $n$ large enough. So
$g^n(W^u(y))$  contains a $(u+1)$-dimension disk which   cross  through $\text{Bl}^u(P_{s+c-2})$. Then we can conclude that $g^n(W^u(y))$  contains a $(u+2)$-strip which is well placed in  $\text{Bl}^u(P_{s+c-2})$ and thus
$W^u(y)\pitchfork W^s(P_{s+c-2})\neq \emptyset$ by the definition of blender.
\end{proof}

Similarly, by induction, we have $W^u(y)\pitchfork W^s(P_j)\neq \emptyset$, where $j=s+c_1+\cdots+c_\kappa$. That is to say, $y\in \Lambda^u$ and we complete the proof of the Lemma.
\end{proof}

We continue to prove Theorem A. By Lemma \ref{lem:2}, there is stably ergodic $g\in \mathcal{V}_c$.   Noting that $\mathcal{V}\cap \text{Diff}^2(M)\subset \mathcal{U}^2(f, \varepsilon)$, we complete the proof.
\end{proof}

\begin{comment}
\vskip 1cm

\flushleft{\bf Yunhua Zhou}\\
College of Mathematics and Statistics,\\
 Chongqing University,
Chongqing, 400030\\
 P. R. China\\
{\em E-mail:} zhouyh@cqu.edu.cn

\end{comment}

%\medskip

\end{document}